\RequirePackage{fix-cm}
\documentclass[twocolumn]{svjour3}          
\smartqed  
\usepackage[numbers]{natbib}
\usepackage{multicol}
\usepackage{float}
\usepackage{textcomp}
\usepackage{amsmath}
\usepackage{amssymb}
\usepackage{graphicx}
\usepackage{algorithm}
\usepackage{algorithmic}
\usepackage{multirow}
\usepackage{booktabs}
\usepackage{rotating}
\newtheorem{coro}[theorem]{\bf Corollary}

\begin{document}
%
%

\title{Scheduling step-deteriorating jobs to minimize the total weighted tardiness on a single machine}
\author{Peng Guo        \and
        Wenming Cheng   \and
        Yi Wang 
}

\institute{P. Guo \at
              School of Mechanical Engineering, Southwest Jiaotong University,
Chengdu, China \\
              \email{pengguo318@gmail.com}           
           \and
           W. Cheng \at
               School of Mechanical Engineering, Southwest Jiaotong University,
Chengdu, China \\
          \email{wmcheng@swjtu.edu.cn}
          \and
          Y. Wang \at
          Corresponding author. Department of Mathematics, Auburn University at Montgomery, AL, USA \\
          \email{ywang2@aum.edu}
}

\date{Received: date / Accepted: date}

\maketitle

\begin{abstract}
This paper addresses the scheduling problem of minimizing the total weighted tardiness on a single machine with step-deteriorating jobs. With the assumption of deterioration, the job processing times are modeled by   step functions of  job starting times and   pre-specified job deteriorating dates.
The introduction of step-deteriorating jobs makes a single machine total weighted tardiness problem more intractable. The computational complexity of this problem under consideration was not determined.
In this study, it is firstly proved to be strongly NP-hard.
Then a mixed integer programming model is derived for solving the problem instances optimally.
In order to tackle  large-sized problems, seven dispatching heuristic procedures are developed for near-optimal solutions. Meanwhile, the solutions delivered by the proposed heuristic are further improved by a pair-wise swap movement.
Computational results are presented to reveal the performance of all proposed approaches.
\end{abstract}

\maketitle\section{Introduction}
Scheduling with meeting job due dates have received increasing attention from managers and researchers since the Just-In-Time concept was introduced in manufacturing facilities. While meeting due dates is only a qualitative performance, it usually implies that time dependent penalties are assessed on late jobs but that no benefits are derived from completing jobs early \citep{Baker2009book}.
In this case, these quantitative scheduling objectives associated with tardiness are naturally highlighted in different manufacturing environments.
The total tardiness and the total weighted tardiness are two of the most common performance measures in tardiness related scheduling problems \citep{Sen2003static}. In the traditional scheduling theory, these  two kinds of single machine scheduling problems have been extensively studied in the literature under the assumption that the processing time of a job is known in advance and constant throughout the entire operation process \citep{Koulamas2010TTU,Vepsalainen1987TWTU,Crauwels1998TWTU,Cheng2005TWTU,Bilge2007TWTU,Wang2009TWTU}.

However,  there are many practical situations that any delay or waiting in starting to process a job may cause to increase   its processing time.
Examples can be found in financial management, steel production, equipment maintenance, medicine treatment and so on.
Such problems are generally known as {\em time dependent scheduling problems} \citep{Gawiejnowicz2008TDS_Book}. Among various time dependent scheduling models, there is one case in which job processing times are formulated by   piecewise defined functions.
In the literature, these jobs with   piecewise defined processing times  are mainly presented by {\em piecewise linear deterioration} and/or {\em step-deterioration}.
In this paper, the single machine total weighted tardiness scheduling problem with step-deteriorating jobs (SMTWTSD) are addressed. For a step-deteriorating job, if it fails to be processed prior to a pre-specified threshold, its processing time will be increased by an extra time; otherwise, it only needs the basic processing time. The corresponding single machine scheduling problem was considered firstly by \citet{Sundararaghavan1994_SPSD} for minimizing the sum of the weighted completion times.

The single machine scheduling problem of makespan minimization under step-deterioration effect was investigated in \citet{Mosheiov1995_SPSD}, and some simple heuristics were introduced and   extended to the setting of multi-machine and multi-step deterioration.
\citet{Jeng2004SPSD} proved that the problem proposed by \citet{Mosheiov1995_SPSD} is NP-hard in the ordinary sense based on a pseudo-polynomial time dynamic programming algorithm, and introduced two dominance rules and a lower bound to develop a branch and bound algorithm for deriving optimal solutions.
\citet{Cheng2001SPSD} showed the total completion time problem with identical job deteriorating dates is NP-complete and introduced a pseudo-polynomial algorithm for the makespan problem. Moreover,
\citet{Jeng2005SPSD} proposed a branch and bound algorithm incorporating a lower bound and two elimination rules for the total completion time problem. Owing to the intractability of the problem,
\citet{He2009SPSD} developed a   branch and bound algorithm  and a weight combination search algorithm to derive the optimal and near-optimal solutions.

The problem was extended by \citet{Cheng2012SPSD_VNS} to the case with parallel machines, where a {\em variable neighborhood search algorithm} was proposed to solve the parallel machine scheduling problem.
Furthermore,
\citet{Layegh2009SPSD} studied the total weighted completion time scheduling problem on a single machine under job step deterioration assumption, and proposed a {\em memetic algorithm} with the average percentage error of 2\%.
Alternatively, batch scheduling with step-deteriorating effect  also attracts the attention of some researchers, for example, referring to \citet{Barketau2008SPSD_batch}, \citet{Leung2008SPSD_batch} and \citet{Mor2012SPSD_batch}.
With regard to the piecewise linear deteriorating model, the single machine scheduling problem with makespan minimization was firstly addressed in \citep{Kunnathur1990SPLD}. Following this line of research, successive research works, such as \citet{Kubiak1998SPLD}, \citet{Cheng2003SPLD}, \citet{Moslehi2010SPLD}, spurred in the literature.

However, these objective functions with due dates were rarely studied under step-deterioration model in the literature.
\citet{Guo2014GVNSSPSDTT} recently found that the total tardiness problem in a single machine is NP-hard, and introduced two heuristic algorithms.
As the more general case, the single machine total weighted tardiness problem (SMTWT) has been extensively studied, and several dispatching heuristics were also proposed for obtaining the near-optimal solutions \citep{Potts1991SPTWT,Kanet2004SPTWT}.
\emph{To the best of our knowledge, there is no research that discusses the} SMTWTSD {\em problem}.  Although the SMTWT problem has been proved to be strongly NP-hard \citep{Lawler1977pp,Lenstra1977complexity}
and \citep[p.~58]{Pin2012}, the complexity of the SMTWTSD problem under consideration is still open. Therefore, this paper gives \emph{the proof of  strong NP-hardness of the} {SMTWTSD} {\em problem}.   \emph{Several efficient dispatching heuristics are presented and analyzed as well}. These dispatching heuristics can deliver a feasible schedule within reasonable computation time for large-sized problem instances. Moreover,  they can be   used to generate an initial solution with certain quality required by a meta-heuristics.

%
%
%
%

The remainder of this paper is organized as follows. Section \ref{sec:Sec_PDF} provides a definition of the single machine total weighted tardiness problem with step-deteriorating jobs and formulates the problem as a mixed integer programming model. The complexity results of the problem considered in this paper and some extended problems are discussed in Section \ref{sec:Sec_CS}.
In Section \ref{sec:Sec_HA}, seven improved heuristic procedures are described. Section \ref{sec:Sec_CE} presents the computational results and analyzes the performances of the proposed heuristics.  Finally, Section \ref{sec:Sec_CN} summarizes the findings of this paper.

\section{Problem description and formulation \label{sec:Sec_PDF}}

The problem considered in this paper is to schedule  $n$ jobs of the set $\mathbb{N}_n:=\{1,2,\ldots,  n\}$, on a single machine for minimizing the  {\em total weighted tardiness}, where the jobs have {\em stepwise} processing times.
%
%
Specifically,
    assume that all jobs are ready at time zero and the machine
is available at all times. Meanwhile, no preemption is assumed. In addition, the machine can handle only
one job at a time, and cannot keep idle until the last job assigned
to it is processed and finished. For each job  $j\in \mathbb{N}_n$, there is a {\em basic
processing time} $a_{j}$, a {\em due date} $d_{j}$ and a given {\em deteriorating threshold}, also called {\em deteriorating
date} $h_{j}$. If the  {\em starting time} $s_{j}$ of job $j\in \mathbb{N}_n$ is less than or
equal to the given threshold $h_{j}$, then job $j$ only requires a basic
processing time $a_{j}$. Otherwise,   an extra
penalty $b_{j}$ is incurred. Thus, the {\em actual processing time} $p_{j}$ of job
$j$ can be defined as a step-function: $p_{j}=a_{j}$ if $s_{j}\leqslant h_{j}$;
$p_{j}=a_{j}+b_{j}$, otherwise. Without loss of generality, the four parameters $a_{j}$, $b_{j}$,
$d_{j}$ and $h_{j}$ are assumed to be positive integers.

Let $\pi=(\pi_{1},\ldots,\pi_{n})$
be a sequence that arranges the current processing order of jobs in $\mathbb{N}_n$, where $\pi_{k}$, $k\in \mathbb{N}_n$,
indicates  the job in position $k$.
The {\em tardiness} $T_{j} $
of  job $j$ in a schedule $\pi$ can be calculated by $$T_{j} =\max\left\{ 0,s_{j} +p_{j} \text{\textminus}d_{j}\right\} .$$
The objective is to find a schedule such that the {\em  total weighted tardiness}
$\sum w_jT_{j}$ is minimized, where the weights $w_j$, $j\in \mathbb{N}_n$ are positive constants.   Using the standard three-field notation \citep{Graham1979optimization}, this
problem studied here can be denoted by $1|p_{j}=a_{j}$ or $a_{j}+b_{j},h_{j}|\sum w_jT_{j}$.

Based on the  above description, we   formulate the problem as a  0-1
integer programming model. Firstly, the decision variable $y_{ij}$, $i,j \in \mathbb{N}_n$ is defined
 such that  $y_{ij}$ is 1 if   job $i$ precedes job
$j$ (not necessarily immediately) on the single-machine, and 0 otherwise.
The formulation of the problem is given below.\\
\emph{Objective function}:
\begin{equation}
\mathrm{minimize}   \quad
Z:=\sum_{j\in \mathbb{N}_n}w_jT_{j}\label{eq:2.1}
\end{equation}
\emph{Subject to}:
\begin{align}
p_{j}=\begin{cases}
a_{j}, & \quad s_{j}\leqslant h_{j}\\
a_{j}+b_{j}, & \quad \mathrm{otherwise},
\end{cases} &\quad \begin{gathered}\forall j\in \mathbb{N}_n\end{gathered}
\label{eq:2.2}\\
s_{i}+p_{i}\leqslant  s_{j}+M(1-y_{ij}),  &\quad \forall i,j\in \mathbb{N}_n,i< j \label{eq:2.3}\\
s_j+p_j \leqslant s_i+My_{ij}, &\quad \forall i,j \in\mathbb{N}_n, i<j \label{eq:2.4}\\
s_{j}+p_{j}-d_{j} \leqslant  T_{j}, &\quad\forall j\in \mathbb{N}_n\label{eq:2.5}\\
y_{ij} \in  \left\{ 0,1\right\}, &\quad\forall i,j\in \mathbb{N}_n,i\neq j\label{eq:2.6}\\
s_{j},T_{j} \geqslant  0, &\quad\forall j\in \mathbb{N}_n, \label{eq:2.7}
\end{align}
where $M$ is a large positive constant such that $M\rightarrow\infty$ as $n\to \infty$.
For example,  $M$ may be  chosen as  $M:=\max_{j\in \mathbb{N}_n}\left\{ d_{j}\right\} +\sum_{j\in \mathbb{N}_n}(a_{j}+b_{j})$.

In the above mathematical model, equation  (\ref{eq:2.1}) represents
the objective of minimizing  the total weighted tardiness. Constraint (\ref{eq:2.2})
defines the processing time of each  job with the consideration of step-deteriorating effect.
Constraint (\ref{eq:2.3})  and \eqref{eq:2.4}
determine the starting
time $s_{j}$ of job $j$ with respect to the decision variables $y_{ij}$.
Constraint (\ref{eq:2.5}) calculates the tardiness of job $j$ with the completion time and the due date.
Finally, Constraints (\ref{eq:2.6}) and (\ref{eq:2.7}) define the
boundary values of variables $y_{ij}$, $s_j$, $T_j$, for $i,j\in \mathbb{N}_n$.

\section{Complexity results \label{sec:Sec_CS}}
This section discusses computational complexity of the problem under consideration. It is generally known that once a problem is proved to be strongly NP-hard, it is impossible to find a polynomial time algorithm or a pseudo polynomial time algorithm to produce its optimal solution. Then heuristic algorithms are presented to obtain near optimal solutions for such a problem. Subsequently, the studied single machine scheduling problem is proved to be strongly NP-hard.
\begin{theorem}\label{thm:NPhard}
The problem $1|p_{j}=a_{j}$ or $a_{j}+b_{j},h_{j}|\sum w_jT_{j}$ is strongly NP-hard.
\end{theorem}

The proof of the theorem is based on reducing 3-PARTITION to the problem $1|p_{j}=a_{j}$ or $a_{j}+b_{j},h_{j}|\sum w_jT_{j}$. For a positive integer $t\in \mathbb{N}$, let  integers $a_j\in \mathbb{N}, j\in \mathbb{N}_{3t}$, and $b\in \mathbb{N}$ such that $\frac{b}{4}<a_j<\frac{b}{2}$, $j\in \mathbb{N}_{3t}$ and $\sum_{j\in \mathbb{N}_{3t}}a_j=tb $. The reduction is based on the following transformation. Let the number of jobs $n=4t-1$. Let the partition jobs be such that for $j\in \mathbb{N}_{3t}$,
\\
\begin{equation}
d_j=0, \quad h_j=tb+2(t-1), \quad b_j=1, \quad w_j=a_j-\frac{b}{4}
\end{equation}
 and \begin{eqnarray}
p_{j}&=&\begin{cases}
a_{j}, & s_{j}\leqslant h_{j},\\
a_{j}+b_{j}, & \mathrm{otherwise}.
\end{cases}
\label{eq:partition}
\end{eqnarray}
Introduce the notation $\mathbb{N}_{m,n}:=\{m,m+1, \ldots, n\}$. Let the $t-1$ enforcer jobs be such that for $j\in \mathbb{N}_{3t+1,4t-1}$,
\begin{eqnarray}
d_j=(j-3t)(b+1), \quad h_j=d_j-1, \quad b_j=1,  \nonumber \\ w_j=(b+b^2)(t^2-t)
\end{eqnarray}
and
\begin{eqnarray}
p_{j}&=&\begin{cases}
1, & s_{j}\leqslant h_{j},\\
1+b_{j}, & \mathrm{otherwise}.
\end{cases}
\label{eq:enforcer}
\end{eqnarray}

The first $3t$ partition jobs are due at time $0$, and the last $(t-1)$ enforcer jobs are due at $b+1$, $2b+2$, $\ldots$ and so on. The deterioration dates of all partition jobs are set at $h_j=tb+2(t-1)$, $j\in \mathbb{N}_{3t}$; while the deterioration dates of all enforcer jobs are set at $d_j-1$, for $j\in \mathbb{N}_{3t+1,4t-1}$, that is, one unit before their individual due dates, respectively.
We introduce the set notation $\mathbb{S}_i:=\{{3i-2}, {3i-1},{3i}\}$ , $i\in \mathbb{N}_t$ for sets of three partition jobs. In general,  assume that for $i\in \mathbb{N}_t$,
\begin{equation}\label{eqn:sumtime}
a_{3i-2}+a_{3i-1}+a_{3i}=b+\delta_i,
\end{equation}
where, due to  $\frac{b}{4}<a_j<\frac{b}{2}$, $j\in \mathbb{N}_{3t}$, we must have
\begin{equation}\label{eqn:deltai}
-\frac{b}{4}<\delta_i < \frac{b}{2}.
\end{equation}
The quantity $\delta_i$, $i\in \mathbb{N}_t$ describes the difference from $b$ of the sum of basic processing times of jobs in $\mathbb{S}_i$ , $i\in \mathbb{N}_t$.
Since $\sum_{j\in \mathbb{N}_{3t}}a_j=tb $, we deduce that
\begin{equation}\label{eqn:delta}
\sum_{i\in \mathbb{N}_t} \delta_i=0.
\end{equation}

For convenience we define $\Delta_0=0$ and for $i\in \mathbb{N}_t$,
$$
\Delta_i=\sum_{j\in \mathbb{N}_i} \delta_j.
$$
That is, $\Delta_i$ is the accumulative difference from $ib$ of the cumulative basic processing times of the partition jobs up to the $i$-th set $\mathbb{S}_i$ of 3 partition jobs. Note by \eqref{eqn:delta}, $\Delta_t=0$.  The following observation is useful.
\begin{lemma}
\begin{equation}\label{eqn:conversion}
\sum_{j\in \mathbb{N}_t} j \delta_j=-\sum_{j\in \mathbb{N}_{t-1}}\Delta_j= -\sum_{j\in \mathbb{N}_{t}}\Delta_j.
\end{equation}
\end{lemma}

\begin{proof}
The proof can be done by direct calculation. The second equality is obvious because  $\Delta_t=0$.
For the first equality, we have
\begin{eqnarray*}
\sum_{j\in \mathbb{N}_t} j \delta_j &= & \delta_1+2\delta_2+\ldots + t\delta_t\\
&=& (\delta_1+\delta_2+\ldots+\delta_t)+(\delta_2+\delta_3+\ldots+\delta_t)\\
&& +\ldots    + (\delta_{t-1}+\delta_t)+\delta_t\\
&=& 0-\delta_1-(\delta_1+\delta_2)-\ldots \\
&&- (\delta_1+\delta_1+\ldots+\delta_{t-2}) \\
&&  -(\delta_1+\delta_2+\ldots+\delta_{t-1})
\end{eqnarray*}
where in the last equality we have used the equality $\sum_{j\in \mathbb{N}_t}\delta_j=0$. Thus we continue by using the definition of $\Delta_j$, $j\in \mathbb{N}_t$ to have
\begin{eqnarray*}
\sum_{j\in \mathbb{N}_t} j \delta_j&=&-\Delta_1-\Delta_2-\ldots-\Delta_{t-2}-\Delta_{t-1}\\
&=&-\sum_{j\in \mathbb{N}_{t-1}}\Delta_j.
\end{eqnarray*}
The lemma is proved.
\end{proof}

Let $\lceil x \rceil$ be the least integer greater than or equal to $x$, and for $k\in \mathbb{N}$, define the index set $J_k:=\{ 3\left\lceil \frac{k}{3}  \right\rceil -2, 3\left\lceil \frac{k}{3}  \right\rceil -1, \ldots,   k\}$.

We are ready to prove the theorem by showing that there exits a schedule with an objective value
\begin{equation}\label{eqn:opt_value}
z^*=\frac{(t^2-t)}{8}(b+b^2)+\sum_{k\in \mathbb{N}_{3t}}\sum_{j\in J_k  }a_j\left(a_k-\frac{b}{4}\right)
\end{equation}
if and only if there exits a solution for the 3-PARTITION problem.

\begin{proof}[Proof of Theorem \ref{thm:NPhard}]

If the 3-PARTITION problem has a solution, the corresponding $3t$ jobs thus can be partitioned into $t$ subsets $\mathbb{S}_i$, $i\in \mathbb{N}_t$, of three jobs each, with the sum of the three processing times in each subset equal to $b$, that is, $\delta_i=0$, for $i\in \mathbb{N}_t$, and  the last $t-1$ jobs are processed exactly during the intervals
$$
[b,b+1], [2 b+1, 2b+2], \ldots, [(t-1)b+t-2, (t-1)b+t-1].
$$
In this scenario, all the $t-1$ enforcer jobs are not tardy and all the $3t$ partition jobs are tardy.   The tardiness of each partition job equals to its completion time. Moreover, no job is deteriorated.

Let $c_j$, $j\in \mathbb{N}_{4t}$ be the completion time of each job.  Let $S_i$ be the starting time of the first job in each set $\mathbb{S}_i$, $i\in \mathbb{N}_t$. When no job  is deteriorated,  that is all jobs are processed with basic processing time, the completion time of each job $j$ in $\mathbb{S}_i$, $i\in \mathbb{N}_t$ is given by $$
c_j=S_i+\sum_{k\in J_j  }a_k.
$$
The  total weighted tardiness of the 3 jobs in $\mathbb{S}_i$, $i\in \mathbb{N}_t$,
equals to
\begin{eqnarray}
 z_i&=& \sum_{j\in \mathbb{S}_i} c_jw_j \nonumber\\
 &=&S_i\sum_{j\in \mathbb{S}_i} w_j + \sum_{j\in \mathbb{S}_i} \sum_{k\in J_j  }a_kw_j   \label{eqn:zk}\\
  &=&[(i-1)b+(i-1)]\frac{b}{4}+\sum_{j\in \mathbb{S}_i}\sum_{k\in J_j  }a_k\left(a_j-\frac{b}{4}\right)  ,  \nonumber
\end{eqnarray}
since $S_i=(i-1)b+(i-1)$ and $\sum_{j\in \mathbb{S}_i} w_j=\frac{b}{4}$, for each $i\in \mathbb{N}_t$.
Thus   the total weighted tardiness is $ \sum_{i\in \mathbb{N}_t}z_i$ which sums to $z^*$ given by \eqref{eqn:opt_value}.

Conversely, if such a 3-partition is not possible, there is at least one $\delta_i\ne 0$, $i\in \mathbb{N}_t$. We next argue that  this must imply  $\Delta z:=z-z^*>0$.

If $\Delta_i>0$, $i\in \mathbb{N}_{t-1}$,  then the $i$-th enforcer job will deteriorate to be processed in the extended time and entail a weighted tardiness.  Introduce the notation $x_+:=\max\{x, 0\}$, and let $\mathcal{N}(i)$, $i\in \mathbb{N}_{t-1}$ denote the number of times the value $\Delta_j>0$, $j\in \mathbb{N}_i$. For convenience, we define $\mathcal{N}(0)=0$. The value  $\mathcal{N}(i)$,  with   $0\le \mathcal{N}(i)\le i$, coincides with the cumulative extended processing time of all the enforcer jobs up to the $i$-th enforcer job.  This implies that the weighted tardiness of the $i$-th enforcer job is given by $$\left(\left(\Delta_i \right)_+ +\mathcal{N}(i)\right)(b+b^2)(t^2-t).$$
 On the other hand, our configuration of the deterioration dates for all the partition jobs ensures that no partition jobs can deteriorate. Therefore, by recalling equation \eqref{eqn:sumtime}, we have
 in this case that for $i\in \mathbb{N}_t$,
$$
S_i=  (i-1)b+(i-1)+\Delta_{i-1} +\mathcal{N}(i-1)
$$
and
$$\sum_{j\in \mathbb{S}_i} w_j=\frac{b}{4}+\delta_i.$$
In view of equation \eqref{eqn:zk} and the weighted tardiness of the enforcer job, we deduce that
the change in the objective function value caused by   $\delta_i$, $i\in \mathbb{N}_t$,    is given by
\begin{eqnarray}
\Delta z_i&=&\left( (i-1)b+(i-1)+\Delta_{i-1} +\mathcal{N}(i-1)  \right)\delta_i \nonumber\\ && +\frac{b}{4}\left(\Delta_{i-1}+\mathcal{N}(i-1)\right) \nonumber \\
& & +\left(\left(\Delta_i \right)_+ +\mathcal{N}(i)\right)(b+b^2)(t^2-t).
\end{eqnarray}

Therefore the total change in the objective function value is given by
\begin{eqnarray*}
\Delta z &=& \sum_{i\in \mathbb{N}_t} \Delta z_i \\
&\ge & (1+b)\sum_{i\in \mathbb{N}_t} (i-1)\delta_i  +\sum_{i\in \mathbb{N}_t} \Delta_{i-1}\left(\delta_i+\frac{b}{4}\right) \\
&& + (b+b^2)(t^2-t) \sum_{i\in \mathbb{N}_t}\mathcal{N}(i) \\
&= & (1+b)\sum_{i\in \mathbb{N}_t} i \delta_i  +\sum_{i\in \mathbb{N}_{t-1}} \Delta_{i}\left(\delta_{i+1}+\frac{b}{4}\right) \\
&& + (b+b^2)(t^2-t) \sum_{i\in \mathbb{N}_{t-1}} \mathcal{N}(i).
\end{eqnarray*}
In the last equality again we have used $\Delta_t=\sum_{j\in \mathbb{N}_t}\delta_j=0$. Now by equation \eqref{eqn:conversion}, we continue to have
\begin{eqnarray}
\Delta z &\ge& (1+b)\sum_{i\in \mathbb{N}_{t-1}} (-\Delta_i)  +\sum_{i\in \mathbb{N}_{t-1}} \Delta_{i}\left(\delta_{i+1}+\frac{b}{4}\right) \nonumber \\
&&+ (b+b^2)(t^2-t) \sum_{i\in \mathbb{N}_{t-1}} \mathcal{N}(i)   \nonumber \\
&=& \sum_{i\in \mathbb{N}_{t-1}} (-\Delta_i) \left( 1+\frac{3}{4}b-\delta_{i+1}  \right)+ \nonumber\\
&& (b+b^2)(t^2-t) \sum_{i\in \mathbb{N}_{t-1}} \mathcal{N}(i). \nonumber\\
\label{eqn:deltaz}
\end{eqnarray}
If there is a $\Delta_i>0$ for some $i\in \mathbb{N}_t$, then $\sum_{i\in \mathbb{N}_{t-1}} \mathcal{N}(i) \ge 1$. Recalling that $\frac{b}{2}>\delta_i>-\frac{b}{4}$, for all $i\in \Delta_t$, we have
\begin{eqnarray*}
\sum_{i\in \mathbb{N}_{t-1}} (-\Delta_i) \left( 1+\frac{3}{4}b-\delta_{i+1}  \right) &>&
(1+b)\sum_{i\in \mathbb{N}_{t-1}} (-\Delta_i)\\
&>& (1+b)\sum_{i\in \mathbb{N}_{t-1}} \left(-i\frac{b}{2}  \right)\\
&=& -\frac{1}{4}b(1+b)t(t-1).
\end{eqnarray*}
Thus in this case, $\Delta z>0$ by equation \eqref{eqn:deltaz}. On the other hand, if all $\Delta_i\le 0$,  and at least one $\delta_i\ne 0$, $i\in \mathbb{N}_t$, then there must be at least one $\Delta_i<0$. Thus
in this case
\begin{eqnarray*}
\Delta z &=&  \sum_{i\in \mathbb{N}_{t-1}} (-\Delta_i) \left( 1+\frac{3}{4}b-\delta_{i+1}  \right)\\
&>&0
\end{eqnarray*}
because $\left( 1+\frac{3}{4}b-\delta_{i+1}  \right)>0 $
for $i\in \mathbb{N}_t$.
Therefore $\Delta z=0$ if and only if $\delta_i=0$, $i\in \mathbb{N}_t$. We have proved the theorem.

\end{proof}
{\bf Remark}: with small changes, the above proof also shows that the scheduling problem of total weighted
tardiness (without deteriorating jobs), represented by $1|$ $|\sum w_jT_{j}$, is strongly NP-hard, of which  proofs might be found in \citep{Lawler1977pp,Lenstra1977complexity}
and \citep[p.~58]{Pin2012}.

After a reflection, we mention that the following three problems also are strongly NP-hard: 1) the problem of total weighted tardiness with deterioration jobs and job release times $r_j$; 2) the problem of total   tardiness with deterioration jobs and job release times $r_j$; 3)  the problem of  maximum lateness of a single machine scheduling problem with job release time $r_j$. Recall that
the maximum lateness is defined to be
$$
L_{\max}=\max\{ L_j:
L_j=c_j-d_j, j\in \mathbb{N}_n\}.
$$
Their proofs can be obtained by slightly modifying our previous proof. In fact, the assumption of release times makes the proof a lot of easier. We summarize these results in the following corollaries.

\begin{coro}
The problem $1|p_{j}=a_{j}$ or \\$a_{j}+b_{j},h_j,r_{j}|\sum w_jT_{j}$ is strongly NP-hard.
\end{coro}
\begin{coro}
The problem $1|p_{j}=a_{j}$ or \\$a_{j}+b_{j},h_j,r_{j}|\sum T_{j}$ is strongly NP-hard.
\end{coro}
\begin{coro}
The problem $1|p_{j}=a_{j}$ or \\$a_{j}+b_{j},h_j,r_{j}|L_{\max}$ is strongly NP-hard.
\end{coro}

\section{Heuristic algorithms\label{sec:Sec_HA}}
The problem under study is proved   strongly NP-hard earlier, then some dispatching heuristics are needed to develop for solving the problem. In this section, the details of these heuristics are discussed. Dispatch heuristics gradually form the whole schedule
by adding one job at a time with the best priority index among the unscheduled jobs. There are several existing  heuristics     designed for the problem without deteriorating jobs. Since the processing times of all jobs considered in our problem depend, respectively, on their starting times, these dispatching heuristics are modified for considering the characteristic of the problem.

Before introducing these procedures, the following notations are defined.
 Let ${N}^s$  denote the {\em ordered set} of  already-scheduled jobs and    ${N}^u$ the {\em unordered set} of  unscheduled jobs.
 Hence $\mathbb{N}_n=  {N}^s\cup  {N}^u $ when the ordering of $N^s$ is not in consideration.
 Let $t^s$ denote the current time, i.e., the maximum completion time of the scheduled jobs in the set $ {N}^s$. We shall call $t^s$ the current time of the sequence $N^s$.
Simultaneously, $t^s$ is also the starting time of the next selected job.
For each of the  unscheduled jobs in $\mathbb{N}^u $, its actual processing time is calculated based on its deteriorating date  and the current time $t^s$.

For most scheduling problems with due dates, the \emph{earliest due date}  (EDD) rule is simple and efficient.
In this paper, the rule is adopted to obtain a schedule by sorting jobs in non-decreasing order of their due dates. In the same way, the \emph{weighted shortest processing time} (WSPT) schedules jobs in decreasing order of $w_j/p_j$. At each iteration, the actual processing times of  unscheduled jobs in ${N}^u$ are needed to recalculate. This is because
when  step-deteriorating effect is considered, the processing time of a job is variable. Even when all jobs are necessarily tardy, the WSPT rule does not guarantee an optimal schedule.
Moreover, the weighted EDD (WEDD) rule introduced by \citet{Kanet2004SPTWT} sequences jobs in non-decreasing order of WEDD, where
\begin{equation}\label{eq:eq_wedd}
  \mathrm{WEDD}_j=d_j/w_j,
\end{equation}

The apparent tardiness cost (ATC) heuristic introduced by \citet{Vepsalainen1987TWTU} was developed for the total weighted tardiness problem when the processing time of a job is constant and known in advance. It showed relatively good performance compared with the EDD and the WSPT. The job with the largest ATC value is selected to be processed. The ATC for job $j$ is determined by the following equation.
\begin{equation}\label{eq:eq_atc}
  \mathrm{ATC}_j=\frac{w_j}{p_j}\exp\left(-\max\{0, d_j-p_j-t^s\}/(\kappa\rho)\right),
\end{equation}
where,  $\kappa$ is a ``look-ahead" parameter usually between 0.5 and 4.5, and $\rho$ is the average processing time of the rest of  unscheduled jobs. The processing time of an already   scheduled job $j\in N^s$ may be $a_j$ or $a_j+b_j$ dependent on if it is deteriorated. Subsequently, the current time $t_s$ is calculated upon the completion of the last job in $N^s$.  The parameter $\rho$ is calculated by averaging  the processing times of   unscheduled jobs in $N^u$  assuming their starting time is at $t^s$.

Based on the cost over time (COVERT) rule \citep{Fisher1976Covert} and the apparent urgency (AU) rule \citep{Morton1984AU}, \citet{Alidaee1996AR} developed a class of heuristic named  COVERT-AU for the standard single machine scheduling weighted tardiness problem. The COVERT-AU heuristic combines the two well known methods, i.e. COVERT and AU.
At the time $t^s$, the COVERT-AU chooses the next job with the largest priority index  $\mathrm{CA}_j$ calculated by the equation
 \begin{equation}\label{eq:eq_ca}
   \mathrm{CA}_j=\frac{w_j}{p_j}(\kappa\rho/(\kappa\rho+\max\{0, d_j-p_j-t^s\})).
\end{equation}
{For the convenience of description, the heuristic with equation \eqref{eq:eq_ca} is denoted by CA in this paper hereafter.

The \emph{weighted modified due date}  (WMDD) rule was developed by \citet{Kanet2004SPTWT} based on modified due date (MDD). In this method, the jobs are processed in non-decreasing order of WMDD. The WMDD is calculated by the   equation
\begin{equation}\label{eq:eq_wmdd}
  \mathrm{WMDD}_j= \frac{1}{w_j} (\max\{p_j, (d_j-t^s)\}).
\end{equation}
Note that when all job weights   are equal, the WMDD is equal to the MDD.

The above heuristics need to recalculate the processing time of the next job for obtaining the priority index except for the EDD and the WEDD. The procedures to recalculate the processing time of the next job is significantly different from   those for  the problem without step-deterioration.   In order to illustrate how  these   heuristics work,
the detailed steps of the WMDD, as an example, are  shown in Algorithm \ref{alg:alg_wmdd}. For   other heuristics, the only difference is the calculation of the priority index.

\begin{algorithm}[htp]
\begin{algorithmic}[1]

\caption{\label{alg:alg_wmdd} The WMDD}

\STATE Input the initial data of a given instance;

\STATE Set ${N}^s=[\;]$, $t^s=0$ and ${N}^u=\{1, 2, \ldots, n\}$;

\STATE Set $k=1$;
\REPEAT
\STATE Compute the processing time of  each job in the set ${N}^u$ based on the current time $t^s$;
\STATE Calculate the WMDD value of each job in $N^u$ according to equation \eqref{eq:eq_wmdd};
\STATE Choose job $j$ from the set ${N}^u$ with the smallest value $\mathrm{WMDD}_j$ to be scheduled in the $k$th position;
\STATE Update the tardiness of job $j$: $T_j=\max\{t^s+p_j-d_j, 0\}$, and ${N}^s=[{N}^s, j]$;
\STATE Delete job $j$ from ${N}^u$;
\STATE $k=k+1$;
\UNTIL {the set ${N}^u$ is empty}
\STATE Calculate the total weighted tardiness of the obtained sequence ${N}^s$.
\end{algorithmic}
\end{algorithm}

A very effective and simple combination search heuristic for minimizing total tardiness was  proposed by \citet{Guo2014GVNSSPSDTT}.
The heuristic is called "Simple Weighted Search Procedure" (SWSP) and works as follows: a combined value of parameters $a_j$, $p_j$ and $h_j$ for job $j$ is calculated as
\begin{equation}\label{eq:eq_swsp}
  m_j=\gamma_1d_j+\gamma_2p_j+\gamma_3h_j,
\end{equation}
where $\gamma_1$, $\gamma_2$ and $\gamma_3$ are three positive constants.
In the SWSP, jobs are sequenced in non-decreasing order of $m$-value.
To accommodate the case of the weighted tardiness, equation \eqref{eq:eq_swsp} is modified to compute a
   priority index $m'$ for a job $j$  calculated by the   equation
\begin{equation}\label{eq:eq_mswsp}
  m'_j=\frac{1}{w_j} (\gamma_1d_j+\gamma_2p_j+\gamma_3h_j).
\end{equation}
The modified method is called "Modified Simple Weighted Search Procedure" (MSWSP).
{In equation \eqref{eq:eq_mswsp}, the values of $\gamma_1$, $\gamma_2$ and $\gamma_3$ are determined by using a dynamically updating strategy. The updating strategy is similar to that proposed by \citet{Guo2014GVNSSPSDTT}. Specifically,  parameter $\gamma_1$ is linearly increased by 0.1 at each iteration and its range is varied from $\gamma_{1\min}$ to $\gamma_{1\max}$.
In this study, $\gamma_{1\min}=0.2$ and $\gamma_{1\max}=0.9$ are chosen based on   preliminary tests by using  randomly generated instances. The parameter $\gamma_2$ adopts a similar  approach with $\gamma_{1\min}$ and $\gamma_{1\max}$ replaced by $\gamma_{2\min}=0.1$ and $\gamma_{2\max}=0.7$, respectively. Once the values of   parameters $\gamma_1$ and $\gamma_2$ are determined, the parameter $\gamma_3:=\max\{1-\gamma_1-\gamma_2, 0.1\}$.}
The detailed steps of the MSWSP is shown in Algorithm \ref{alg:MSWSP}.

\begin{algorithm}[htp]
\begin{algorithmic}[1]
\caption{\label{alg:MSWSP} The MSWSP }

\STATE Input the initial data of a given instance;

\STATE Set $c_{[0]}=0$, $N^{u}=\{1,\cdots,n\}$ and ${N}^s=[\;]$;

\STATE Generate the entire set $\Omega$ of possible triples of weights. For each triple $(\omega_{1}, \omega_{2}, \omega_{3})\in \Omega$,
perform the following steps;

\STATE Choose  job $i$ with minimal due date to be scheduled in
the first position;

\STATE $\:$$c_{[1]}=c_{[0]}+a_{i}$, ${N}^s=\left[ {N}^s,\; i\right]$;

\STATE $\:$delete  job $i$ from $N^{u}$;

\STATE $\:$set $k=2$;

\REPEAT

\STATE choose  job $j$ from $N^{u}$ with the smallest value $m'_j$ to
be scheduled in the $k$th position;
\IF { $c_{[k-1]}>h_{j}$}

\STATE $c_{[k]}=c_{[k-1]}+a_{j}+b_{j}$;

\ELSE

\STATE $c_{[k]}=c_{[k-1]}+a_{j}$;

\ENDIF

\STATE ${N}^s=\left[ {N}^s,\;j\right]$;

\STATE delete  job $j$ from $N^{u}$;

\UNTIL {the set $N^{u}$ is empty}.

\STATE Calculate the   total weighted tardiness  of the obtained schedule $ {N}^s$;

\STATE Output the finial solution $ {N}^s$.
\end{algorithmic}
\end{algorithm}
%
%
%

In order to further improve the quality of the near-optimal solutions, a pairwise swap movement (PS) is incorporated into these heuristics. Let $N^s$ be the sequence output by a heuristic. A swap operation chooses a pair of jobs in positions $i$ and $j$, $1\le i,j\le n$,  from the sequence $ {N}^s$, and exchanges their positions. Denote the new sequence by $N^s_{ji}$. Subsequently,    the total weighted tardiness of the   sequence  $N^s_{ji}$ is calculated. If the  new sequence $N^s_{ji}$ is better with a smaller tardiness than the incumbent one, the incumbent one is replaced by the new sequence. The swap   operation is repeated for any combination of   two indices $i$ and $j$, where $1\leqslant i < j \leqslant n$. Thus the size of the pairwise swap movement (PS) is $n(n-1)/2$. In the following, a heuristic   with the PS
movement is denoted by the symbol $\mathrm{ALG}_{\mathrm{PS}}$, where ALG is one of the above mentioned heuristic algorithms.
For example,  $\mathrm{EDD}_{\mathrm{PS}}$ represents that the earliest due date rule is applied first , then the solution obtained by the EDD is further improved by the PS movement.

\section{Computational experiments\label{sec:Sec_CE}}
In this section, the computational experiments and results are presented to analyze the performance of the above dispatching heuristics. Firstly, randomly generated test problem instances varying from small to large sizes are described. Next, preliminary experiments are carried out to determine appropriate values for the parameters used in some of the heuristics. Then, a comparative analysis of all seven dispatching heuristics is performed. Furthermore, the results of the best method are compared with optimal solutions delivered by ILOG CPLEX 12.5 for small-sized problem instances. All heuristics were coded in MATLAB 2010 and run on a personal computer with Pentium Dual-Core E5300 2.6 GHz processor and 2 GB of RAM.

\subsection{Experimental design\label{sec:subsec_ED}}
The problem instances were generated using the method proposed by \citet{Guo2014GVNSSPSDTT} as follows. For each job $j$, a basic processing time $a_j$ was generated from the uniform distribution [1, 100], a weight $w_j$ was generated from the uniform distribution [1, 10], and a deteriorating penalty $b_j$ is generated from the uniform distribution [1, 100$\times \tau$], where $\tau=0.5$. Problem hardness is likely to depend on the value ranges of deteriorating dates and due dates. For each job $j$, a deteriorating date $h_j$ was drawn from the uniform distribution over three intervals $H_1$:=[1, \emph{A}/2], $H_2$:=[\emph{A}/2, \emph{A}] and $H_3$:=[1, \emph{A}], where $A=\sum_{j \in \mathbb{N}_n}a_j$. Meanwhile, a due date $d_j$ was generated from the uniform distribution
 [$C'_{\max}(1-T-R/2), C'_{\max}(1-T+R/2)$], where $C'_{\max}$ is the value of the maximum completion time obtained by scheduling the jobs in the non-decreasing order of the ratios $a_j/b_j$, $j\in \mathbb{N}_n$,
 $T$ is the average tardiness factor and $R$ is the relative range of due dates. Both  $T$ and $R$ were set at 0.2, 0.4, 0.6, 0.8 and 1.0.

 Overall, 75 different combinations   were generated for different $h$, $T$ and $R$. For the purpose of obtaining   optimal solutions, the number of jobs in each instance was taken to be one of the two sizes of 8 and 10. For the heuristics, the number of jobs can be varied from the small sizes to the large sizes, that comprises 14 sizes including 8, 10, 15, 20, 25, 30, 40, 50, 75, 100, 250, 500, 750 and 1000. In each combination of $h$, $T$, $R$, and $n$, 10 replicates were generated and solved.
 Thus, there are 750 instances for each problem size, totalling 10500 problem instances, which are available from
http:$ //$www.researchgate.net$/$profile$/$Peng\_Guo9.

 In general, the performances of a heuristic is measured by the average relative percentage deviation (RPD)  of the heuristic solution values from   optimal solutions value or   best solution values.
 The average RPD value is calculated as $\frac{1}{K} \sum_{k=1}^{K}\frac{Z^k_{\mathrm{alg}}-Z^k_{\mathrm{opt}}}{Z^k_{\mathrm{opt}}}\times100$, where $K$ is the number of problem instances and $Z^k_{\mathrm{alg}}$ and $Z^k_{\mathrm{opt}}$ are the objective function value of the heuristic method and the optimal solution value for instance $k$, respectively.

 The objective function value delivered by a heuristics may be equal to 0 for an instance  with a low tardiness factor $T$ and a high relative range   $R$ of due dates.  The zero objective value means that all jobs are finished on time. It is troublesome to obtain the RPD in this case, since necessarily $Z^k_{\mathrm{opt}}=0$, thus it leads to a  division by 0 that is undefined.  To avoid this situation, in this paper, the relative improvement versus the worst result (RIVW) used by \citet{Valente2012SMWQTSP} is adopted to evaluate the performance of a proposed heuristics.
 For a given instance,   the RIVW for a   heuristic is defined by the following way. Let $Z_{\mathrm{best}}$ and $Z_{  \mathrm{worst}}$ denote the best and worst solution values delivered by all considered heuristics in comparison, respectively.  When
  $Z_{ \mathrm{best}}=Z_{ \mathrm{worst}}$, the RIVW value of a heuristic algorithm is set to 0. Otherwise, the RIVW  value is calculated as $$\mathrm{RIVW}=(Z_{ \mathrm{worst}}-Z_{ \mathrm{alg}})/Z_{ \mathrm{worst}}\times 100.$$  Based on the definition of RIVW, it can be observed that the bigger the  RIVW value, the better the quality of the corresponding solution.

%
%

\subsection{Parameter selection\label{sec:subsec_PS}}
In order to select an appropriate value for the parameter $\kappa$, preliminary tests were conducted on a separate problem set, which contains instances with 20, 50, 100, and 500 jobs. For each of  these job sizes $n$, 5 replicates were produced for each combination of $n$, $h$, $T$ and $R$. Subsequently, for each problem instance, the objective function value was obtained by using all considered seven heuristics   with a candidate value of $\kappa$. Then the  results of all problem instances were analyzed to determine the best  value of $\kappa$.
The candidate values of the parameter $\kappa$ are  chosen to be {0.5, 1.0, \ldots, 4.5}, which are usually used in the ATC and the CA  for   traditional single machine problem \citep{Vepsalainen1987TWTU,Alidaee1996AR}. The computational tests show that the solutions delivered by the ATC and the CA with $\kappa=0.5$ are relatively better compared with the results obtained by  other values of $\kappa$. Thus, $\kappa$ is set to 0.5 in the sequel.

\subsection{Experimental results\label{sec:subsec_CR}}
The computational results of the proposed seven heuristics are listed in Table \ref{tab:tab_CR}. Specifically, this table provides the mean RIVW values for each heuristic, as well as the number of instances with the best solution ($\mathrm{Num}_{\mathrm{best}}$) found by a heuristic method  from 750 instances  for each problem size. Each mean RIVW value for a heuristic and a particular problem size in Table \ref{tab:tab_CR} is the average of RIVW values from the 750 instances for a given problem size.    From the table, the EDD and WEDD rules are clearly outperformed by the WSPT, ATC, CA, WMDD, and MSWSP heuristics. The mean RIVW values delivered by EDD and WEDD are significantly less than that given by  other heuristics. This is due to the fact that the EDD rule only considers   job  due dates, while the WEDD   relies only on weights  and due dates. In addition, the two rules do not consider the effect of   step-deterioration in calculating the priority index. It is worthwhile to note that the results achieved by the WEDD is better than that given by the EDD.

As far as  the remaining methods, the RIVW values obtained by the WSPT and MSWSP heuristics  are worse than that of   other three (ATC, CA and WMDD) heuristics. But the performance of the MSWSP is better than the WSPT. This indicates that the MSWSP can produce good results for our problem, but it fails to obtain   better solutions compared with the three improved methods (ATC, CA and WMDD).

There is no significant difference between the  results produced by the ATC, CA and WMDD heuristics. The CA procedure provides slightly higher mean relative improvement versus the worst result values. For large-sized instances, the number of the best solutions achieved by the CA is much higher than the ATC and the WMDD.
Therefore, {\em the CA procedure can be deemed as the best one among the seven dispatching heuristic algorithms}.

Computational times of these dispatching heuristic algorithms for each job size are listed in Table \ref{tab:tab_ctime1}. As the size of  instances increases, the CPU time of all methods grows at different degrees. Totally, the MSWSP consumes the most time compared with other algorithms, but surprisingly its maximum CPU time is only 5.71 seconds for the intractable instance with 1000 jobs. The average CPU time of the CA which is 0.66 seconds is less than that of the MSWSP. Since the EDD and the WEDD mainly depend on the ranking index of all jobs' due dates, their computational times are less than 0.01 seconds. It was observed that the CPU time of the other five algorithms follow almost the same trend.

\begin{sidewaystable*}[htp]\scriptsize
  \centering
  \caption{Computational results of the dispatching heuristic algorithms}
  \setlength{\tabcolsep}{5pt}
    \begin{tabular}{rrrrrrrrrrrrrrr}
    \toprule
    $n$     & \multicolumn{7}{c}{RIVW(\%)}                          & \multicolumn{7}{c}{$\mathrm{Num}_{\mathrm{best}}$} \\
    \cmidrule[0.05em](r){2-8}
    \cmidrule[0.05em](lr){9-15}

          & EDD   & WSPT  & WEDD  & ATC   & CA    & WMDD  & MSWSP & EDD   & WSPT  & WEDD  & ATC   & CA    & WMDD  & MSWSP \\
          \midrule
    8     & 15.56  & 39.87  & 28.63  & 48.50  & 55.01  & 49.06  & 50.40  & 87    & 202   & 32    & 331   & 375   & 307   & 206  \\
    10    & 15.38  & 40.76  & 26.66  & 49.31  & 57.68  & 52.36  & 50.57  & 70    & 149   & 10    & 299   & 367   & 278   & 145  \\
    15    & 16.97  & 40.15  & 28.23  & 55.60  & 62.25  & 57.82  & 53.01  & 78    & 80    & 4     & 274   & 340   & 249   & 104  \\
    20    & 17.89  & 41.35  & 27.50  & 57.73  & 65.56  & 60.42  & 54.13  & 70    & 72    & 2     & 237   & 363   & 232   & 79  \\
    25    & 16.92  & 41.11  & 27.60  & 58.49  & 67.46  & 62.92  & 53.26  & 67    & 55    & 2     & 206   & 373   & 241   & 72  \\
    30    & 17.39  & 41.47  & 27.41  & 59.93  & 68.64  & 64.49  & 52.52  & 66    & 61    & 0     & 190   & 388   & 215   & 68  \\
    40    & 18.44  & 41.10  & 26.38  & 61.81  & 70.31  & 66.22  & 52.84  & 77    & 46    & 1     & 166   & 383   & 242   & 78  \\
    50    & 18.76  & 40.59  & 27.31  & 62.62  & 71.22  & 67.17  & 52.30  & 71    & 30    & 0     & 183   & 400   & 237   & 71  \\
    75    & 19.00  & 41.34  & 27.41  & 63.65  & 72.03  & 68.59  & 51.70  & 71    & 31    & 0     & 165   & 421   & 247   & 71  \\
    100   & 20.13  & 41.07  & 27.73  & 64.49  & 72.81  & 69.09  & 51.86  & 75    & 36    & 0     & 143   & 428   & 254   & 75  \\
    250   & 19.54  & 41.37  & 27.46  & 65.91  & 73.74  & 69.47  & 49.92  & 78    & 42    & 0     & 124   & 472   & 250   & 78  \\
    500   & 19.55  & 41.47  & 27.42  & 66.47  & 74.01  & 69.42  & 49.29  & 68    & 34    & 0     & 116   & 487   & 245   & 68  \\
    750   & 19.46  & 41.54  & 27.37  & 66.44  & 74.08  & 69.30  & 49.12  & 69    & 35    & 0     & 105   & 499   & 243   & 69  \\
    1000  & 19.50  & 41.46  & 27.32  & 66.38  & 74.13  & 69.32  & 48.92  & 67    & 32    &       & 95    & 502   & 250   & 67  \\
    Total &       &       &       &       &       &       &       & 1014  & 905   & 51    & 2634  & 5798  & 3490  & 1251  \\
    Avg.  & 18.18  & 41.05  & 27.46  & 60.52  & 68.50  & 63.97  & 51.42  & 72    & 65    & 4     & 188   & 414   & 249   & 89  \\
    \bottomrule
    \end{tabular}%
  \label{tab:tab_CR}%
\end{sidewaystable*}

\begin{sidewaystable*}[htp]\footnotesize
  \centering
  \caption{Computational times of dispatching heuristic algorithms}

    \begin{tabular}{cccccccc}
    \toprule
    $n$     & \multicolumn{7}{c}{CPU Time(s)} \\
 \cmidrule[0.05em](r){2-8}
          & EDD   & WSPT  & WEDD  & ATC   & CA    & WMDD  & MSWSP \\
          \midrule
    8     & $<$0.01 & $<$0.01 & $<$0.01 & $<$0.01 & $<$0.01 & $<$0.01 & 0.01  \\
    10    & $<$0.01 & $<$0.01 & $<$0.01 & $<$0.01 & $<$0.01 & $<$0.01 & 0.01  \\
    15    & $<$0.01 & $<$0.01 & $<$0.01 & $<$0.01 & $<$0.01 & $<$0.01 & 0.02  \\
    20    & $<$0.01 & $<$0.01 & $<$0.01 & $<$0.01 & $<$0.01 & $<$0.01 & 0.03  \\
    25    & $<$0.01 & $<$0.01 & $<$0.01 & $<$0.01 & $<$0.01 & $<$0.01 & 0.04  \\
    30    & $<$0.01 & $<$0.01 & $<$0.01 & $<$0.01 & $<$0.01 & $<$0.01 & 0.05  \\
    40    & $<$0.01 & $<$0.01 & $<$0.01 & $<$0.01 & $<$0.01 & $<$0.01 & 0.06  \\
    50    & $<$0.01 & $<$0.01 & $<$0.01 & $<$0.01 & 0.01  & $<$0.01 & 0.08  \\
    75    & $<$0.01 & 0.01  & $<$0.01 & 0.01  & 0.02  & 0.01  & 0.14  \\
    100   & $<$0.01 & 0.02  & $<$0.01 & 0.02  & 0.05  & 0.02  & 0.20  \\
    250   & $<$0.01 & 0.12  & $<$0.01 & 0.14  & 0.30  & 0.14  & 0.65  \\
    500   & $<$0.01 & 0.47  & $<$0.01 & 0.58  & 1.19  & 0.56  & 1.81  \\
    750   & $<$0.01 & 1.11  & $<$0.01 & 1.33  & 2.72  & 1.31  & 3.48  \\
    1000  & $<$0.01 & 2.05  & $<$0.01 & 2.42  & 4.91  & 2.39  & 5.67  \\
          &       &       &       &       &       &       &  \\
    Avg.  & $<$0.01 & 0.27  & $<$0.01 & 0.32  & 0.66  & 0.32  & 0.88  \\
    \bottomrule
    \end{tabular}%

  \label{tab:tab_ctime1}%
\end{sidewaystable*}%

{Subsequently, the solutions delivered by the seven dispatching heuristics are improved by the PS movement. The seven heuristics with the PS movement are denoted by
$\mathrm{EDD}_{\mathrm{PS}}$, $\mathrm{WSPT}_{\mathrm{PS}}$, $\mathrm{WEDD}_{\mathrm{PS}}$, $\mathrm{ATC}_{\mathrm{PS}}$, $\mathrm{CA}_{\mathrm{PS}}$, $\mathrm{WMDD}_{\mathrm{PS}}$  and $\mathrm{MSWSP}_{\mathrm{PS}}$, respectively.
A comparison of these methods is given in Table \ref{tab:tab_cre_ps}.
Again, this table lists the mean relative improvement versus the worst result (RIVW) for each algorithm  and the number of instances with the best solution found by each of the seven algorithms with the PS movement.}

From Table \ref{tab:tab_cre_ps}, it can be observed that  the $\mathrm{CA}_{\mathrm{PS}}$ provides the best performance among these procedures. In fact, the $\mathrm{CA}_{\mathrm{PS}}$ not only gives the largest RIVW value, but also gives a better   solution for most of the instances. For medium- and large-sized instances, the $\mathrm{CA}_{\mathrm{PS}}$ shows   better performance in terms of the number of best solution ($\mathrm{Num}_{\mathrm{best}}$) compared with the other six methods. In particular,
for the case with 1000 jobs,
the $\mathrm{CA}_{\mathrm{PS}}$ gives the best solutions for 538 over the 750 instances.
It is found that $\mathrm{CA}_{\mathrm{PS}}$ delivers best solutions for on average 473 out of 750 instances for all job sizes.
The average RIVW values delivered by the $\mathrm{WEDD}_{\mathrm{PS}}$, $\mathrm{ATC}_{\mathrm{PS}}$, $\mathrm{CA}_{\mathrm{PS}}$, $\mathrm{WMDD}_{\mathrm{PS}}$ and $\mathrm{MSWSP}_{\mathrm{PS}}$ are more than 40\%.
The RIVW value of the WSPT is only 15.86\%,  significantly less than  that achieved by the other six methods.

Computational times of these methods are listed in Table \ref{tab:tab_tab_ct_ps}.
The average computational times of the seven methods are very close, and the gap of the average CPU times between these  methods is not more than one second.
As expected, the CPU times of these algorithms are increased as the number of jobs increases.
But the computational times of the seven methods are not more than 80 seconds even for the 1000-job case.

{In order to further analyze the results, the one-way Analysis of Variance (ANOVA) is used to check whether the observed difference in the RIVW values for  the dispatching heuristics with the PS movement are statistically significant.
The $\mathrm{WSPT}_{\mathrm{PS}}$ is removed from the statistical analysis since it is clearly worse than the remaining ones.
The means plot and the Fisher Least Significant Difference (LSD) intervals at the 95\% confidence level are shown in Figure \ref{fig:fig_h_ps}.
If the  LSD intervals of two algorithms are not overlapped, the performances of the tested algorithms are  statistically significantly different. Otherwise, the performances of the two algorithms do not lie significantly in the difference.
As it can be seen, {\em $\mathrm{ATC}_{\mathrm{PS}}$, $\mathrm{CA}_{\mathrm{PS}}$ and $\mathrm{WMDD}_{\mathrm{PS}}$
 are not statistically different because their confidence intervals are overlapped.} This observation is really important since it gives the conclusion that can not be obtained from a table of average RIVW results. Moreover, the $\mathrm{CA}_{\mathrm{PS}}$ and $\mathrm{WMDD}_{\mathrm{PS}}$ are statistically significantly better than the other three methods ($\mathrm{EDD}_{\mathrm{PS}}$, $\mathrm{WEDD}_{\mathrm{PS}}$ and $\mathrm{MSWSP}_{\mathrm{PS}}$) by having their LSD intervals of the RIVW values higher than those of other methods.
 However, the  $\mathrm{WEDD}_{\mathrm{PS}}$ and the  $\mathrm{MSWSP}_{\mathrm{PS}}$ are not statistically different due to their overlapping confidence intervals.}

{To evaluate the effect of the PS movement, a comparison of the CA heuristic with the $\mathrm{CA}_{\mathrm{PS}}$ procedure is provided in Table \ref{tab:tab_CA_CA_PS}.
Table \ref{tab:tab_CA_CA_PS} gives the mean relative improvements versus the worst result values for the two procedures, as well as the number of times the $\mathrm{CA}_{\mathrm{PS}}$ procedure performs better than ($\mathrm{Num}_{\mathrm{better}}$) or equal  the CA ($\mathrm{Num}_{\mathrm{equal}}$). Since the heuristic CA always gives   worse result when comparing   with the $\mathrm{CA}_{\mathrm{PS}}$, its RIVW values in Table \ref{tab:tab_CA_CA_PS} equal to zero for all instances.
On average, the  RIVW values obtained by the $\mathrm{CA}_{\mathrm{PS}}$ are 24.71\% better than the RIVW values achieved by the CA.
These results show that {\em the PS movement can significantly improve the quality of   solutions delivered by the CA heuristic for most of the instances.}}

\begin{sidewaystable*}[htp]\scriptsize
  \centering
  \caption{Computational results of dispatching heuristic algorithms with the PS movement}
  \setlength{\tabcolsep}{4pt}
    \begin{tabular}{ccccccccccccccc}
    \toprule
    $n$     & \multicolumn{7}{c}{RIVW(\%)}                          & \multicolumn{7}{c}{$\mathrm{Num}_{\mathrm{best}}$} \\

       \cmidrule[0.05em](r){2-8}
    \cmidrule[0.05em](lr){9-15}
          & $\mathrm{EDD}_{\mathrm{PS}}$ & $\mathrm{WSPT}_{\mathrm{PS}}$  & $\mathrm{WEDD}_{\mathrm{PS}}$  & $\mathrm{ATC}_{\mathrm{PS}}$  & $\mathrm{CA}_{\mathrm{PS}}$  & $\mathrm{WMDD}_{\mathrm{PS}}$  & $\mathrm{MSWSP}_{\mathrm{PS}}$  & $\mathrm{EDD}_{\mathrm{PS}}$ & $\mathrm{WSPT}_{\mathrm{PS}}$  & $\mathrm{WEDD}_{\mathrm{PS}}$  & $\mathrm{ATC}_{\mathrm{PS}}$  & $\mathrm{CA}_{\mathrm{PS}}$  & $\mathrm{WMDD}_{\mathrm{PS}}$  & $\mathrm{MSWSP}_{\mathrm{PS}}$ \\
           \midrule
    8     & 26.95  & 12.84  & 29.26  & 31.67  & 33.12  & 31.90  & 32.54  & 243   & 441   & 288   & 506   & 535   & 514   & 422  \\
    10    & 28.76  & 14.37  & 32.27  & 35.18  & 36.57  & 35.79  & 35.02  & 162   & 339   & 179   & 460   & 483   & 468   & 335  \\
    15    & 33.13  & 15.04  & 36.71  & 40.43  & 42.39  & 40.40  & 39.80  & 148   & 233   & 140   & 392   & 435   & 389   & 218  \\
    20    & 35.00  & 15.81  & 38.67  & 43.41  & 44.91  & 43.64  & 42.11  & 139   & 202   & 110   & 370   & 416   & 370   & 181  \\
    25    & 35.88  & 14.58  & 40.10  & 44.61  & 46.77  & 44.86  & 42.66  & 131   & 166   & 98    & 346   & 426   & 358   & 160  \\
    30    & 36.92  & 14.06  & 41.30  & 46.29  & 47.81  & 46.46  & 43.76  & 134   & 181   & 113   & 333   & 423   & 351   & 158  \\
    40    & 37.79  & 16.21  & 42.93  & 48.17  & 50.03  & 48.53  & 45.07  & 135   & 174   & 104   & 318   & 420   & 365   & 134  \\
    50    & 37.99  & 16.08  & 43.40  & 48.96  & 50.79  & 49.45  & 45.68  & 130   & 154   & 98    & 305   & 441   & 320   & 133  \\
    75    & 38.47  & 16.16  & 44.23  & 50.22  & 52.24  & 50.65  & 46.34  & 137   & 164   & 100   & 272   & 460   & 304   & 139  \\
    100   & 38.97  & 15.78  & 44.96  & 51.10  & 53.00  & 51.70  & 46.68  & 136   & 153   & 101   & 264   & 471   & 291   & 133  \\
    250   & 38.98  & 16.94  & 45.44  & 51.99  & 54.21  & 52.49  & 47.02  & 141   & 159   & 101   & 234   & 527   & 254   & 140  \\
    500   & 38.87  & 17.72  & 45.43  & 52.20  & 54.56  & 52.66  & 46.95  & 141   & 161   & 103   & 244   & 520   & 263   & 140  \\
    750   & 38.58  & 18.45  & 45.36  & 51.88  & 54.51  & 52.50  & 46.82  & 144   & 166   & 100   & 252   & 523   & 254   & 143  \\
    1000  & 38.37  & 18.02  & 45.20  & 51.67  & 54.52  & 52.44  & 46.64  & 147   & 166   & 105   & 240   & 538   & 252   & 146  \\
    Total &       &       &       &       &       &       &       & 2068  & 2859  & 1740  & 4536  & 6618  & 4753  & 2582  \\
    Avg.  & 36.05  & 15.86  & 41.09  & 46.27  & 48.24  & 46.68  & 43.36  & 148   & 204   & 124   & 324   & 473   & 340   & 184  \\

    \bottomrule
    \end{tabular}%
  \label{tab:tab_cre_ps}%
\end{sidewaystable*}

\begin{sidewaystable*}[htbp]\footnotesize
  \centering
  \caption{Computational times of dispatching heuristic algorithms with the PS movement}
    \begin{tabular}{cccccccc}
    \toprule
    \multirow{2}[1]{*}{$n$}      & \multicolumn{7}{c}{CPU Time(s)} \\
 \cmidrule[0.05em](r){2-8}
          & $\mathrm{EDD}_{\mathrm{PS}}$ & $\mathrm{WSPT}_{\mathrm{PS}}$  & $\mathrm{WEDD}_{\mathrm{PS}}$  & $\mathrm{ATC}_{\mathrm{PS}}$  & $\mathrm{CA}_{\mathrm{PS}}$  & $\mathrm{WMDD}_{\mathrm{PS}}$  & $\mathrm{MSWSP}_{\mathrm{PS}}$ \\
\midrule
    8     & $<$0.01 & $<$0.01 & $<$0.01 & $<$0.01 & $<$0.01 & $<$0.01 & 0.01  \\
    10    & $<$0.01 & $<$0.01 & $<$0.01 & $<$0.01 & $<$0.01 & $<$0.01 & 0.01  \\
    15    & $<$0.01 & $<$0.01 & $<$0.01 & $<$0.01 & $<$0.01 & $<$0.01 & 0.02  \\
    20    & $<$0.01 & $<$0.01 & $<$0.01 & $<$0.01 & $<$0.01 & $<$0.01 & 0.03  \\
    25    & $<$0.01 & $<$0.01 & $<$0.01 & $<$0.01 & 0.01  & $<$0.01 & 0.04  \\
    30    & $<$0.01 & 0.01  & $<$0.01 & 0.01  & 0.01  & 0.01  & 0.05  \\
    40    & 0.01  & 0.02  & 0.01  & 0.02  & 0.02  & 0.02  & 0.08  \\
    50    & 0.02  & 0.03  & 0.02  & 0.03  & 0.04  & 0.03  & 0.11  \\
    75    & 0.07  & 0.08  & 0.07  & 0.08  & 0.10  & 0.08  & 0.21  \\
    100   & 0.14  & 0.16  & 0.14  & 0.17  & 0.19  & 0.16  & 0.34  \\
    250   & 1.47  & 1.56  & 1.47  & 1.60  & 1.75  & 1.60  & 2.12  \\
    500   & 9.71  & 9.87  & 9.79  & 10.10  & 10.66  & 10.12  & 11.56  \\
    750   & 30.3  & 30.28  & 30.35  & 30.96  & 32.06  & 30.96  & 33.69  \\
    1000  & 69.43 & 68.93  & 69.61  & 70.29  & 72.10  & 70.27  & 74.75  \\
          &       &       &       &       &       &       &  \\
    Avg.  & 7.94  & 7.92  & 7.96  & 8.09  & 8.35  & 8.09  & 8.79  \\

    \bottomrule
    \end{tabular}%
  \label{tab:tab_tab_ct_ps}%
\end{sidewaystable*}%

\begin{figure*}[htbp]
  \centering
  \includegraphics[width=0.75\textwidth]{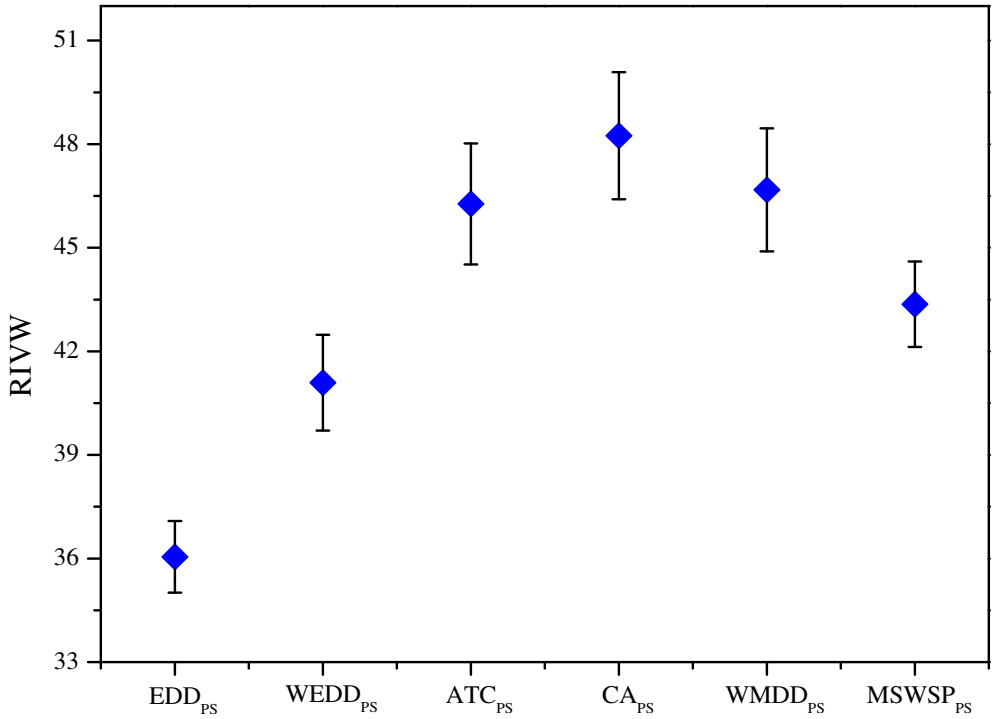}\\
  \caption{Means plot and the LSD intervals (at the 95\% confidence level) for the different dispatching heuristic algorithms with the PS movement}\label{fig:fig_h_ps}
\end{figure*}

\begin{table}[htbp]\footnotesize
  \centering
  \caption{Comparison of CA with $\mathrm{CA}_{\mathrm{PS}}$}

    \begin{tabular}{ccccc}
    \toprule
    \multirow{2}[3]{*}{$n$} & \multicolumn{2}{c}{RIVW(\%)} & \multicolumn{2}{c}{$\mathrm{CA}_{\mathrm{PS}}$ versus CA} \\
     \cmidrule[0.05em](r){2-5}
          & CA    & $\mathrm{CA}_{\mathrm{PS}}$ & $\mathrm{Num}_{\mathrm{equal}}$ & $\mathrm{Num}_{\mathrm{better}}$ \\
          \midrule
    8     & 0.00  & 18.15  & 225   & 525  \\
    10    & 0.00  & 19.56  & 171   & 579  \\
    15    & 0.00  & 24.22  & 80    & 670  \\
    20    & 0.00  & 25.73  & 59    & 691  \\
    25    & 0.00  & 26.08  & 44    & 706  \\
    30    & 0.00  & 26.09  & 47    & 703  \\
    40    & 0.00  & 27.92  & 46    & 704  \\
    50    & 0.00  & 26.74  & 55    & 695  \\
    75    & 0.00  & 27.63  & 66    & 684  \\
    100   & 0.00  & 28.09  & 72    & 678  \\
    250   & 0.00  & 25.07  & 103   & 647  \\
    500   & 0.00  & 24.40  & 110   & 640  \\
    750   & 0.00  & 23.45  & 117   & 633  \\
    1000  & 0.00  & 22.86  & 119   & 631  \\
    Avg.  & 0.00  & 24.71  & 94    & 656  \\
    \bottomrule
    \end{tabular}%

  \label{tab:tab_CA_CA_PS}%
\end{table}%

Next, the $\mathrm{CA}_{\mathrm{PS}}$ heuristic is compared with the optimal solutions delivered by the CPLEX 12.5 for instances with 8 jobs and 10 jobs. Here, the performance of the heuristic is measured by the mean relative improvement of the optimum objective function value versus the heuristic solution (RIVH), as well as the number of instances with the optimal solution  given by the heuristic ($\mathrm{Num}_{\mathrm{opt}}$). For a given instance, the relative improvement of the optimum objective function value versus heuristic is calculated as follows.
When $Z_{\mathrm{opt}}=Z_{\mathrm{alg}}$, the RIVH value is set to 0. Otherwise, the RIVH value is calculated as $$\mathrm{RIVH}=(Z_{\mathrm{alg}}-Z_{\mathrm{opt}})/Z_{\mathrm{alg}}\times100.$$
According to the definition of  the RIVH value,   the smaller the mean RIVH value, the better the quality of the solutions delivered by a heuristic is for the set of instances.

The comparison results were shown in Table \ref{tab:tab_cos1} and \ref{tab:tab_cos2}. The two tables show that 
{\em the  tardiness factor $T$ and the deteriorating interval $H$ can significantly affect the performance of the heuristic $\mathrm{CA}_{\mathrm{PS}}$.} For large values of $T$, the RIVH values are relatively small and the corresponding objective function values are on average quite close to the optimum achieved by the CPLEX. When the deteriorating interval is $H_2$,  most jobs tend to have  large deteriorating dates, and may   not be  deteriorated.
Thus the RIVH values of the test instances with $H_2$ are less than that of the instances with $H_1$ and $H_3$.
This observation has been demonstrated by \citet{Cheng2012SPSD_VNS} for parallel machine scheduling problem.
 In this case, the heuristic $\mathrm{CA}_{\mathrm{PS}}$ can give more optimal solutions compared with the instances with $H_1$ and $H_3$. Overall, the $\mathrm{CA}_{\mathrm{PS}}$ is effective in solving the problem under consideration since the maximum mean RIVH value is below 20\% for instances with 10 jobs.

\begin{table}[htbp]\footnotesize
  \centering
  \caption{Comparison with optimum results for instances with 8 jobs.}
    \begin{tabular}{rrrrrrrr}
    \toprule
    \multicolumn{1}{c}{\multirow{2}[4]{*}{$T$}} & \multicolumn{1}{c}{\multirow{2}[4]{*}{$R$}} & \multicolumn{3}{c}{RIVH(\%)} & \multicolumn{3}{c}{$\mathrm{Num}_{\mathrm{opt}}$} \\
   \cmidrule[0.05em](r){3-5}
   \cmidrule[0.05em](r){6-8}
    \multicolumn{1}{c}{} & \multicolumn{1}{c}{} & \multicolumn{1}{c}{$H_1$} & \multicolumn{1}{c}{$H_2$} & \multicolumn{1}{c}{$H_3$} & \multicolumn{1}{c}{$H_1$} & \multicolumn{1}{c}{$H_2$} & \multicolumn{1}{c}{$H_3$} \\
     \midrule
   0.2   & \multicolumn{1}{c}{0.2} & \multicolumn{1}{c}{40.39 } & \multicolumn{1}{c}{1.07 } & \multicolumn{1}{c}{16.94 } & \multicolumn{1}{c}{0 } & \multicolumn{1}{c}{7 } & \multicolumn{1}{c}{5 } \\
          & \multicolumn{1}{c}{0.4} & \multicolumn{1}{c}{48.40 } & \multicolumn{1}{c}{0.47 } & \multicolumn{1}{c}{34.66 } & \multicolumn{1}{c}{3 } & \multicolumn{1}{c}{8 } & \multicolumn{1}{c}{4 } \\
          & \multicolumn{1}{c}{0.6} & \multicolumn{1}{c}{9.46 } & \multicolumn{1}{c}{7.97 } & \multicolumn{1}{c}{35.52 } & \multicolumn{1}{c}{7 } & \multicolumn{1}{c}{7 } & \multicolumn{1}{c}{5 } \\
          & \multicolumn{1}{c}{0.8} & \multicolumn{1}{c}{18.78 } & \multicolumn{1}{c}{10.00 } & \multicolumn{1}{c}{0.00 } & \multicolumn{1}{c}{7 } & \multicolumn{1}{c}{9 } & \multicolumn{1}{c}{10 } \\
          & \multicolumn{1}{c}{1} & \multicolumn{1}{c}{9.29 } & \multicolumn{1}{c}{4.56 } & \multicolumn{1}{c}{6.25 } & \multicolumn{1}{c}{7 } & \multicolumn{1}{c}{8 } & \multicolumn{1}{c}{8 } \\
    0.4   & \multicolumn{1}{c}{0.2} & \multicolumn{1}{c}{13.50 } & \multicolumn{1}{c}{2.32 } & \multicolumn{1}{c}{14.23 } & \multicolumn{1}{c}{1 } & \multicolumn{1}{c}{7 } & \multicolumn{1}{c}{2 } \\
          & \multicolumn{1}{c}{0.4} & \multicolumn{1}{c}{7.00 } & \multicolumn{1}{c}{0.00 } & \multicolumn{1}{c}{7.50 } & \multicolumn{1}{c}{3 } & \multicolumn{1}{c}{10 } & \multicolumn{1}{c}{3 } \\
          & \multicolumn{1}{c}{0.6} & \multicolumn{1}{c}{12.88 } & \multicolumn{1}{c}{4.62 } & \multicolumn{1}{c}{4.48 } & \multicolumn{1}{c}{3 } & \multicolumn{1}{c}{6 } & \multicolumn{1}{c}{8 } \\
          & \multicolumn{1}{c}{0.8} & \multicolumn{1}{c}{33.70 } & \multicolumn{1}{c}{7.99 } & \multicolumn{1}{c}{14.46 } & \multicolumn{1}{c}{4 } & \multicolumn{1}{c}{6 } & \multicolumn{1}{c}{5 } \\
          & \multicolumn{1}{c}{1} & \multicolumn{1}{c}{10.30 } & \multicolumn{1}{c}{9.01 } & \multicolumn{1}{c}{18.17 } & \multicolumn{1}{c}{8 } & \multicolumn{1}{c}{5 } & \multicolumn{1}{c}{3 } \\
    0.6   & \multicolumn{1}{c}{0.2} & \multicolumn{1}{c}{8.50 } & \multicolumn{1}{c}{1.31 } & \multicolumn{1}{c}{1.90 } & \multicolumn{1}{c}{3 } & \multicolumn{1}{c}{6 } & \multicolumn{1}{c}{7 } \\
          & \multicolumn{1}{c}{0.4} & \multicolumn{1}{c}{9.94 } & \multicolumn{1}{c}{0.98 } & \multicolumn{1}{c}{8.97 } & \multicolumn{1}{c}{2 } & \multicolumn{1}{c}{6 } & \multicolumn{1}{c}{4 } \\
          & \multicolumn{1}{c}{0.6} & \multicolumn{1}{c}{6.38 } & \multicolumn{1}{c}{0.50 } & \multicolumn{1}{c}{2.27 } & \multicolumn{1}{c}{4 } & \multicolumn{1}{c}{5 } & \multicolumn{1}{c}{5 } \\
          & \multicolumn{1}{c}{0.8} & \multicolumn{1}{c}{5.84 } & \multicolumn{1}{c}{2.55 } & \multicolumn{1}{c}{10.09 } & \multicolumn{1}{c}{5 } & \multicolumn{1}{c}{7 } & \multicolumn{1}{c}{1 } \\
          & \multicolumn{1}{c}{1} & \multicolumn{1}{c}{13.31 } & \multicolumn{1}{c}{0.03 } & \multicolumn{1}{c}{4.58 } & \multicolumn{1}{c}{3 } & \multicolumn{1}{c}{8 } & \multicolumn{1}{c}{5 } \\
    0.8   & \multicolumn{1}{c}{0.2} & \multicolumn{1}{c}{3.43 } & \multicolumn{1}{c}{0.36 } & \multicolumn{1}{c}{3.43 } & \multicolumn{1}{c}{7 } & \multicolumn{1}{c}{7 } & \multicolumn{1}{c}{5 } \\
          & \multicolumn{1}{c}{0.4} & \multicolumn{1}{c}{9.58 } & \multicolumn{1}{c}{0.14 } & \multicolumn{1}{c}{1.09 } & \multicolumn{1}{c}{3 } & \multicolumn{1}{c}{8 } & \multicolumn{1}{c}{6 } \\
          & \multicolumn{1}{c}{0.6} & \multicolumn{1}{c}{4.17 } & \multicolumn{1}{c}{0.12 } & \multicolumn{1}{c}{1.71 } & \multicolumn{1}{c}{3 } & \multicolumn{1}{c}{9 } & \multicolumn{1}{c}{5 } \\
          & \multicolumn{1}{c}{0.8} & \multicolumn{1}{c}{3.32 } & \multicolumn{1}{c}{0.87 } & \multicolumn{1}{c}{0.48 } & \multicolumn{1}{c}{4 } & \multicolumn{1}{c}{7 } & \multicolumn{1}{c}{7 } \\
          & \multicolumn{1}{c}{1} & \multicolumn{1}{c}{3.25 } & \multicolumn{1}{c}{1.06 } & \multicolumn{1}{c}{1.85 } & \multicolumn{1}{c}{5 } & \multicolumn{1}{c}{6 } & \multicolumn{1}{c}{7 } \\
    1     & \multicolumn{1}{c}{0.2} & \multicolumn{1}{c}{0.77 } & \multicolumn{1}{c}{0.13 } & \multicolumn{1}{c}{0.78 } & \multicolumn{1}{c}{8 } & \multicolumn{1}{c}{7 } & \multicolumn{1}{c}{6 } \\
          & \multicolumn{1}{c}{0.4} & \multicolumn{1}{c}{1.16 } & \multicolumn{1}{c}{0.00 } & \multicolumn{1}{c}{2.35 } & \multicolumn{1}{c}{6 } & \multicolumn{1}{c}{9 } & \multicolumn{1}{c}{3 } \\
          & \multicolumn{1}{c}{0.6} & \multicolumn{1}{c}{4.02 } & \multicolumn{1}{c}{0.00 } & \multicolumn{1}{c}{2.36 } & \multicolumn{1}{c}{4 } & \multicolumn{1}{c}{9 } & \multicolumn{1}{c}{6 } \\
          & \multicolumn{1}{c}{0.8} & \multicolumn{1}{c}{2.79 } & \multicolumn{1}{c}{0.00 } & \multicolumn{1}{c}{1.48 } & \multicolumn{1}{c}{4 } & \multicolumn{1}{c}{9 } & \multicolumn{1}{c}{6 } \\
          & \multicolumn{1}{c}{1} & \multicolumn{1}{c}{0.75 } & \multicolumn{1}{c}{0.13 } & \multicolumn{1}{c}{1.28 } & \multicolumn{1}{c}{6 } & \multicolumn{1}{c}{8 } & \multicolumn{1}{c}{5 } \\

    Avg.  &       & 11.24  & 2.25  & 7.87  & 4.40  & 7.36  & 5.24  \\
    \bottomrule
    \end{tabular}%
  \label{tab:tab_cos1}%
\end{table}%

\begin{table}[htbp]\footnotesize
  \centering
  \caption{Comparison with optimum results for instances with 10 jobs}
    \begin{tabular}{rrrrrrrr}
    \toprule
    \multicolumn{1}{c}{\multirow{2}[4]{*}{$T$}} & \multicolumn{1}{c}{\multirow{2}[4]{*}{$R$}} & \multicolumn{3}{c}{RIVH(\%)} & \multicolumn{3}{c}{$\mathrm{Num}_{\mathrm{opt}}$} \\
      \cmidrule[0.05em](r){3-5}
   \cmidrule[0.05em](r){6-8}
    \multicolumn{1}{c}{} & \multicolumn{1}{c}{} & \multicolumn{1}{c}{$H_1$} & \multicolumn{1}{c}{$H_2$} & \multicolumn{1}{c}{$H_3$} & \multicolumn{1}{c}{$H_1$} & \multicolumn{1}{c}{$H_2$} & \multicolumn{1}{c}{$H_3$} \\
    \midrule

    0.2   & \multicolumn{1}{c}{0.2} & \multicolumn{1}{c}{55.29 } & \multicolumn{1}{c}{2.52 } & \multicolumn{1}{c}{38.39 } & \multicolumn{1}{c}{2 } & \multicolumn{1}{c}{7 } & \multicolumn{1}{c}{3 } \\
          & \multicolumn{1}{c}{0.4} & \multicolumn{1}{c}{70.00 } & \multicolumn{1}{c}{20.40 } & \multicolumn{1}{c}{30.71 } & \multicolumn{1}{c}{3 } & \multicolumn{1}{c}{5 } & \multicolumn{1}{c}{5 } \\
          & \multicolumn{1}{c}{0.6} & \multicolumn{1}{c}{31.10 } & \multicolumn{1}{c}{5.64 } & \multicolumn{1}{c}{16.71 } & \multicolumn{1}{c}{5 } & \multicolumn{1}{c}{8 } & \multicolumn{1}{c}{8 } \\
          & \multicolumn{1}{c}{0.8} & \multicolumn{1}{c}{25.00 } & \multicolumn{1}{c}{16.07 } & \multicolumn{1}{c}{17.36 } & \multicolumn{1}{c}{7 } & \multicolumn{1}{c}{8 } & \multicolumn{1}{c}{5 } \\
          & \multicolumn{1}{c}{1} & \multicolumn{1}{c}{14.85 } & \multicolumn{1}{c}{22.23 } & \multicolumn{1}{c}{31.82 } & \multicolumn{1}{c}{7 } & \multicolumn{1}{c}{6 } & \multicolumn{1}{c}{5 } \\
    0.4   & \multicolumn{1}{c}{0.2} & \multicolumn{1}{c}{18.28 } & \multicolumn{1}{c}{0.88 } & \multicolumn{1}{c}{10.54 } & \multicolumn{1}{c}{1 } & \multicolumn{1}{c}{7 } & \multicolumn{1}{c}{4 } \\
          & \multicolumn{1}{c}{0.4} & \multicolumn{1}{c}{20.75 } & \multicolumn{1}{c}{2.59 } & \multicolumn{1}{c}{27.79 } & \multicolumn{1}{c}{1 } & \multicolumn{1}{c}{5 } & \multicolumn{1}{c}{0 } \\
          & \multicolumn{1}{c}{0.6} & \multicolumn{1}{c}{43.98 } & \multicolumn{1}{c}{7.06 } & \multicolumn{1}{c}{21.03 } & \multicolumn{1}{c}{0 } & \multicolumn{1}{c}{5 } & \multicolumn{1}{c}{3 } \\
          & \multicolumn{1}{c}{0.8} & \multicolumn{1}{c}{35.02 } & \multicolumn{1}{c}{16.17 } & \multicolumn{1}{c}{23.44 } & \multicolumn{1}{c}{2 } & \multicolumn{1}{c}{3 } & \multicolumn{1}{c}{3 } \\
          & \multicolumn{1}{c}{1} & \multicolumn{1}{c}{16.40 } & \multicolumn{1}{c}{14.17 } & \multicolumn{1}{c}{28.16 } & \multicolumn{1}{c}{0 } & \multicolumn{1}{c}{4 } & \multicolumn{1}{c}{4 } \\
    0.6   & \multicolumn{1}{c}{0.2} & \multicolumn{1}{c}{10.27 } & \multicolumn{1}{c}{1.27 } & \multicolumn{1}{c}{4.09 } & \multicolumn{1}{c}{1 } & \multicolumn{1}{c}{7 } & \multicolumn{1}{c}{4 } \\
          & \multicolumn{1}{c}{0.4} & \multicolumn{1}{c}{7.45 } & \multicolumn{1}{c}{5.21 } & \multicolumn{1}{c}{5.36 } & \multicolumn{1}{c}{5 } & \multicolumn{1}{c}{3 } & \multicolumn{1}{c}{5 } \\
          & \multicolumn{1}{c}{0.6} & \multicolumn{1}{c}{3.01 } & \multicolumn{1}{c}{1.55 } & \multicolumn{1}{c}{5.14 } & \multicolumn{1}{c}{2 } & \multicolumn{1}{c}{4 } & \multicolumn{1}{c}{4 } \\
          & \multicolumn{1}{c}{0.8} & \multicolumn{1}{c}{7.70 } & \multicolumn{1}{c}{3.44 } & \multicolumn{1}{c}{5.92 } & \multicolumn{1}{c}{2 } & \multicolumn{1}{c}{5 } & \multicolumn{1}{c}{3 } \\
          & \multicolumn{1}{c}{1} & \multicolumn{1}{c}{8.51 } & \multicolumn{1}{c}{2.71 } & \multicolumn{1}{c}{5.88 } & \multicolumn{1}{c}{1 } & \multicolumn{1}{c}{3 } & \multicolumn{1}{c}{3 } \\
    0.8   & \multicolumn{1}{c}{0.2} & \multicolumn{1}{c}{2.31 } & \multicolumn{1}{c}{0.53 } & \multicolumn{1}{c}{0.75 } & \multicolumn{1}{c}{4 } & \multicolumn{1}{c}{7 } & \multicolumn{1}{c}{5 } \\
          & \multicolumn{1}{c}{0.4} & \multicolumn{1}{c}{8.01 } & \multicolumn{1}{c}{0.03 } & \multicolumn{1}{c}{3.61 } & \multicolumn{1}{c}{1 } & \multicolumn{1}{c}{8 } & \multicolumn{1}{c}{3 } \\
          & \multicolumn{1}{c}{0.6} & \multicolumn{1}{c}{3.76 } & \multicolumn{1}{c}{1.03 } & \multicolumn{1}{c}{2.81 } & \multicolumn{1}{c}{3 } & \multicolumn{1}{c}{3 } & \multicolumn{1}{c}{5 } \\
          & \multicolumn{1}{c}{0.8} & \multicolumn{1}{c}{3.52 } & \multicolumn{1}{c}{0.29 } & \multicolumn{1}{c}{2.05 } & \multicolumn{1}{c}{5 } & \multicolumn{1}{c}{7 } & \multicolumn{1}{c}{2 } \\
          & \multicolumn{1}{c}{1} & \multicolumn{1}{c}{4.66 } & \multicolumn{1}{c}{1.06 } & \multicolumn{1}{c}{3.66 } & \multicolumn{1}{c}{3 } & \multicolumn{1}{c}{6 } & \multicolumn{1}{c}{3 } \\
    1     & \multicolumn{1}{c}{0.2} & \multicolumn{1}{c}{0.67 } & \multicolumn{1}{c}{0.18 } & \multicolumn{1}{c}{1.14 } & \multicolumn{1}{c}{6 } & \multicolumn{1}{c}{6 } & \multicolumn{1}{c}{4 } \\
          & \multicolumn{1}{c}{0.4} & \multicolumn{1}{c}{2.66 } & \multicolumn{1}{c}{0.20 } & \multicolumn{1}{c}{1.61 } & \multicolumn{1}{c}{2 } & \multicolumn{1}{c}{7 } & \multicolumn{1}{c}{4 } \\
          & \multicolumn{1}{c}{0.6} & \multicolumn{1}{c}{2.93 } & \multicolumn{1}{c}{0.13 } & \multicolumn{1}{c}{0.92 } & \multicolumn{1}{c}{1 } & \multicolumn{1}{c}{9 } & \multicolumn{1}{c}{5 } \\
          & \multicolumn{1}{c}{0.8} & \multicolumn{1}{c}{1.44 } & \multicolumn{1}{c}{0.01 } & \multicolumn{1}{c}{0.92 } & \multicolumn{1}{c}{4 } & \multicolumn{1}{c}{9 } & \multicolumn{1}{c}{5 } \\
          & \multicolumn{1}{c}{1} & \multicolumn{1}{c}{1.37 } & \multicolumn{1}{c}{0.02 } & \multicolumn{1}{c}{1.51 } & \multicolumn{1}{c}{4 } & \multicolumn{1}{c}{9 } & \multicolumn{1}{c}{3 } \\

    Avg.  &       & 15.96  & 5.01  & 11.65  & 2.88  & 6.04  & 3.92  \\
    \bottomrule
    \end{tabular}%
  \label{tab:tab_cos2}%
\end{table}%

\section{Conclusions\label{sec:Sec_CN}}
The total weighted tardiness as the general case has   more important meaning in the practical situation.
In this paper, the single machine scheduling problem with step-deteriorating jobs for minimizing the total weighted tardiness was addressed.
Based on the
characteristics of this problem,  a mathematical programming model is presented for obtaining the optimal solution, and,  the proof of the strong NP-hardness for the problem under consideration is given.
Afterwards, seven heuristics are designed to obtain the near-optimal solutions for  randomly generated problem instances.
Computational results show that these dispatching heuristics can deliver relatively good solutions at   low cost of computational time. Among these dispatching heuristics, the CA procedure as the best solution method can quickly generate a good schedule even for large instances.
Moreover, the test results  clearly indicate that these methods can be significantly improved by the
pairwise swap movement.

In the future,
the consideration of developing meta-heuristics such as a  genetic algorithm or ant colony optimization approach might be an interesting issue. For   medium-sized problems, it is possible that a meta-heuristic  could give better solutions within reasonable computational time. Another consideration is to investigate the total weighted tardiness problem with the step-deteriorating effects under other machine settings, such as parallel machines or flow-shops.
%
%
%
%
%

\begin{acknowledgements}

This work is supported by the National Natural Science
Foundation of China (No. 51405403) and the Fundamental Research Funds
for the Central Universities (No. 2682014BR019).
\end{acknowledgements}

\bibliographystyle{spbasic}
 \newcommand{\noop}[1]{}

\end{document}